\documentclass{amsart}
\usepackage{bbm}
\usepackage{graphicx}
\usepackage{mathrsfs}
\usepackage{amssymb}
\usepackage{mathtools}

\usepackage{color}

\newtheorem{thm}{Theorem}

\newenvironment{thmbis}[1]
  {\renewcommand{\thethm}{\ref{#1}$'$}%
   \addtocounter{thm}{-1}%
   \begin{thm}}
  {\end{thm}}

\newtheorem{cor}[thm]{Corollary}
\newtheorem{lem}[thm]{Lemma}
\newtheorem{prop}[thm]{Proposition}
\theoremstyle{definition}
\newtheorem{defn}[thm]{Definition}
\theoremstyle{remark}
\newtheorem{rem}[thm]{Remark}
\numberwithin{equation}{subsection}

\newcommand{\Sl}{\mathrm{SL}_n}
\newcommand{\A}{\mathbb{A}_f}

\newcommand{\M}{\mathrm{Mat}_n}

\newcommand{\R}{\mathbb{R}}

\newcommand{\Gl}{\mathrm{GL}_n}
\newcommand{\Q}{\mathbb{Q}}
\newcommand{\PP}{\mathbb{P}}
\newcommand{\Zh}{\hat{\mathbb{Z}}}
\newcommand{\Z}{\mathbb{Z}}

\begin{document}

\title[The $\Gl$-Connes-Marcolli Systems]
 {The $\Gl$-Connes-Marcolli Systems}

\author{Yunyi Shen}

\address{Department of Mathematics,FSU}

\email{yshen2@math.fsu.edu}









\begin{abstract}
In this paper, we generalize the results of Laca, Larsen, and Neshveyev on the $\mathrm{GL}_2$-Connes-Marcolli system to the $\Gl$ systems. We introduce the $\Gl$-Connes-Marcolli systems and discuss the question
of the existence and the classification of KMS equilibrium states at different inverse temperatures $\beta$.
In particular, using an ergodicity argument, we prove that in the range $n-1<\beta\leq n$, there is only one
KMS state. We show that there are no KMS states for $\beta < n-1$ and not an integer, while we construct KMS
states for integer values of $\beta$ in the range $1\leq \beta \leq n-1$, and we classify extremal KMS states
for $\beta > n$.
\end{abstract}

\maketitle

\section{Introduction}

Two decades ago, Bost and Connes introduced a C*-dynamical system, which we refer to here as the BC-system (see \cite{BB}). This system plays an important role in relating quantum statistical mechanics and number theory. It has the Riemann Zeta function as its partition function. It shows a deep relation between the symmetries and the equilibrium states of the quantum statistical mechanical system and the class field theory.
The values of the ground equilibrium states of the system on a subalgebra of
 rational observables span the maximal abelian extension of the rational number field $\Q$. This is related to the primary case of the Hilbert's 21th problem, see \cite{BB} and Chapter 3 of \cite{CM2}.
The dual system of the BC system is also used to study the Riemann Hypothesis, (see \cite{Co} and
Chapter 4 of \cite{CM2}).
Because of the success of the BC system, several generalizations were later constructed. In particular,
Connes and Marcolli constructed the $\mathrm{GL}_2$ system (which we refer to here as the CM system) in \cite{CM}. Using the properties of this system, Connes, Marcolli and Ramachandran constructed another quantum statistical mechanical system, which has the same properties of the BC system for imaginary quadratic fields,
and relates to another known case of Hilbert's 12th problem: the construction of the maximal abelian extension
of imaginary quadratic fields in terms of values of modular functions at CM points (see \cite{CMR1}, \cite{CMR2}).
The remaining known cases of the Hilbert's 12th problem, namely abelian extensions of CM fields,
have also been interpreted from the quantum statistical mechanical point of view in \cite{Yalk}.
Since then, people attempted to construct many similar (BC-like) systems. Among them, Ha and Paugam
gave an abstract construction of systems by using Shimura varieties, through which BC-like systems
can be built for any number field (see \cite{HP}).
As we mentioned above, the most significant properties of the BC like systems are reflected by the action of symmetric group of the system on the equilibrium states.
These states are described by the KMS (Kubo-Martin-Schwinger)-condition. They are called the KMS states.
It is important to study the phase transitions of a system, and how the set of KMS states changes
by the effect of phase transitions.

In \cite{LLN}, Laca, Larsen and Neshveyev developed a method based on Hecke operators and ergodic theory and gave a thorough study of the KMS states over the ($\mathrm{GL}_2$-) Connes-Marcolli system. At the end of their paper, they suggested that their results could be generalized to a $\Gl$-Connes-Marcolli system. In this paper, we develop a similar set of tools, with Hecke operators and ergodic theory, adapted to the
$\Gl$-setup and we prove the corresponding results of the $\Gl$-Connes-Marcolli systems.

In section 2, we introduce the $\Gl$-Connes-Marcolli systems as suggested in \cite{LLN}. As usual, those systems are constructed through group actions and related groupoids. Then we analyse the fix points of the group action. Finally, we establish a one-to-one correspondence between the set of KMS states and the set of Borel measures with certain scaling and normalizing conditions, see Theorem \ref{thmKMS}.

Next, in section 3, we first recall some basic concepts of Hecke algebras and some integration formulas related to the Hecke operators, which will play a major role in the proof the main theorem, Theorem \ref{main}. In 3.2, the main part of this paper, we study the phase transition problem and generalize the results in \cite{LLN} to the $\Gl$ system cases. The correspondence between the KMS states and the Borel measures allows us to study the properties of the measures instead of directly studying the KMS states. We prove a Key Lemma (Lemma \ref{Key}) and a corollary (Corollary \ref{Keycor}), which show that for probability measures on ${\rm PGL}_n^+(\R)\times {\rm Mat}_n(\Q_p)$
satisfying a certain $\beta$-dependent
scaling condition and normalization, the set ${\rm PGL}_n^+(\R)\times {\rm GL}_n(\Q_p)$ is
of full measure, provided $\beta$ is not equal to an integer between $0$ and $n-1$.
We then show that for  inverse temperatures $\beta<n-1$ and $\beta\neq 0,1,\ldots,n-1$ there is no KMS$_\beta$ state. We also prove the main theorem: in the range $\beta\in(n-1,n]$, there can be only one KMS$_\beta$ state for the $\Gl$-Connes-Marcolli system (if one exists at all). In 3.3, we show existence of a
KMS$_\beta$ state in this range $\beta\in(n-1,\,n]$, by explicitly showing
how it can be constructed. We also discuss the existence of KMS states at the dividing points
$\beta=0,1,\ldots,n-1$ and we give a classification of the extremal KMS states in the range $\beta>n$.

\section{The $\Gl$-Connes-Marcolli Systems}
Let us first fix some notations. Let $\PP=\mathrm{PGL}_n^+(\R)$ and $\Gamma=\Sl(\Z)$. Let
$Y=\PP\times\M(\Zh)$ and $X=\PP\times\M(\A)$ where $\A$ is the finite Adele ring over $\Q$.
Let $\Gamma\backslash\Gl^+(\Q)\times_\Gamma X$ be the quotient space of $\Gl^+(\Q)\times X$ by the $\Gamma\times\Gamma$-action:
$$(\gamma_1,\gamma_2)(g,x)=(\gamma_1g\gamma^{-1}_2,\gamma_2x)\quad\mathrm{with}\,(g,x)\in\Gl^+(\Q)\times X\,\mathrm{and}\,\gamma_{1,2}\in\Gamma.$$
The subspace $\Gamma\backslash\Gl^+(\Q)\boxtimes_\Gamma Y$ of $\Gamma\backslash\Gl^+(\Q)\times_\Gamma X$ is the space whose elements can be represented by pairs $(g,y)\in\Gl^+(\Q)\times Y$ such that $gy\in Y$.

We construct an involutive associative algebra and a Hilbert space representation associated to
the space $\Gamma\backslash\Gl^+(\Q)\boxtimes_\Gamma Y$, which is analogous to the
algebra of the ${\rm GL}_2$ system constructed in \cite{CM}, see also \S 5.2 of \cite{CM2}.

For $f_1, f_2\in C_c(\Gamma\backslash\Gl^+(\Q)\boxtimes_\Gamma Y)$, there is naturally a convolution defined as
\begin{equation}\label{convprod}
f_1\ast f_2(g,y)=\sum\limits_{h\in\Gamma\backslash\Gl^+(\Q), hy\in Y}f_1(gh^{-1},hy)f_2(h,y),
\end{equation}
and there is also an involution $$f_1^*(g,y)=\overline{f_1(g^{-1},gy)}.$$
If we set $G_y=\{g\in\Gl^+(\Q)\mid\,gy\in Y\}$, then the action $$\pi_y(f)\delta(g)=\sum\limits_{h\in\Gamma\backslash G_y}f(gh^{-1},hy)\delta(h)$$ gives a representation of $C_c(\Gamma\backslash\Gl^+(\Q)\boxtimes_\Gamma Y)$ on the Hilbert space $\mathfrak{H}_y=\ell^2(\Gamma\backslash G_y)$. Then one defines the norm as in Definition~3.43 of \cite{CM2} as
$$\|f\|=\sup\limits_{y\in Y}\{\|\pi_y(f)\|\}$$ on $C_c(\Gamma\backslash\Gl^+(\Q)\boxtimes_\Gamma Y)$. The completion of $C_c(\Gamma\backslash\Gl^+(\Q)\boxtimes_\Gamma Y)$ in this norm is a C* algebra, which is denoted by $\mathrm{C}_r^*(\Gamma\backslash\Gl^+(\Q)\boxtimes_\Gamma Y)$.

\begin{defn} \label{GLnalg}
The $\Gl$-Connes-Marcolli algebra is the C*-algebra $$\mathcal{A}=\mathrm{C}_r^*(\Gamma\backslash\Gl^+(\Q)\boxtimes_\Gamma Y).$$
\end{defn}

Recall that a time evolution on a $C^*$-algebra $\mathcal{A}$ is a continuous one-parameter family of
automorphisms $\sigma: {\mathbb R} \to {\rm Aut}(\mathcal{A})$.

\begin{lem}\label{sigmatlem}
Setting $\sigma_t(f)(g,x)=\mathrm{det}(g)^{it}f(g,x)$, for $f\in C_c(\Gamma\backslash \Gl^+(\mathbb{Q})\boxtimes_\Gamma Y)$ determines a time evolution $\sigma_t$ on the $C^*$-algebra $\mathcal{A}$
of Definition~\ref{GLnalg}.
\end{lem}

\proof In general, for an algebra of the form $C_c(\Gamma\backslash \Gl^+(\mathbb{Q})\boxtimes_\Gamma Y)$,
with the convolution product \eqref{convprod}, any group homomorphism $N:\Gl^+(\mathbb{Q}) \to \R^*_+$ with
the property that $\Gamma \subset {\rm Ker}(N)$ determines a time evolution by setting $\sigma_t(f)(g,x)=N(g)^{it} f(g,x)$,
see  Section 2 of \cite{LLN}.
Indeed, this property suffices to ensure that, with respect to the convolution product \eqref{convprod}, one has
$\sigma_t(f_1 \ast f_2)=  \sigma_t(f_1)\ast \sigma_t(f_2)$. Clearly setting $N(g)=\det(g)$ has the desired properties.
\endproof

Recall that an element $a\in\mathcal{A}$ is said to be {\em entire} if the function $t\mapsto\sigma_t(a)$ can be extended to an entire function on the complex number field $\mathbb{C}$, see Definition 2.5.20 of \cite{BR}. We also recall that a KMS$ _\beta$ state is a $\sigma_t$ invariant state $\varphi$ over $\mathcal{A}$ such that, for all entire elements $a,b$ of $\mathcal{A}$, the
relation $\varphi(ab)=\varphi(b\sigma_{i\beta}(a))$ holds, see Definition 5.3.1 of \cite{BR2}.

If we identify the space $\Gamma\backslash\Gl^+(\Q)\times_\Gamma\PP\times\{0\}$ with the space $\Gamma\backslash\Gl^+(\Q)\times_\Gamma\PP$, then the restriction of the *-algebra structure and representation of $$C_c(\Gamma\backslash\Gl^+(\Q)\boxtimes_\Gamma Y)$$ on $$ C_c(\Gamma\backslash\Gl^+(\Q)\times_\Gamma\PP\times\{0\})$$ makes $$C_c(\Gamma\backslash\Gl^+(\Q)\times_\Gamma\PP)$$ a normed *-algebra. Then the completion of $$C_c(\Gamma\backslash\Gl^+(\Q)\times_\Gamma\PP)$$ in the induced norm, which is denoted by $$\mathcal{B}=\mathrm{C}_r^*(\Gamma\backslash\Gl^+(\Q)\times_\Gamma\PP),$$ is also a $C^*$ algebra. In a similar way, the restriction of $\sigma_t$ to $$C_c(\Gamma\backslash\Gl^+(\Q)\times_\Gamma\PP\times\{0\})$$ defines a time evolution on $\mathcal{B}$. So $(\mathcal{B},\sigma_t)$ also forms a $C^*$-dynamical system.

Given a matrix of the form
$$g=\begin{pmatrix}k\\&\ddots\\&&k\end{pmatrix}$$ with $k\in\mathbb{N}$,
consider the function $u_g$ that takes values $u_g(h,x)=1$ if $h\in\Gamma.g$ and $u_g(h,x)=0$ otherwise.
This defines a unitary multiplier $u_g$ of $\mathcal{B}$ which  is an eigenfunction of $\sigma_t$, namely such that
$$\sigma_t(u_g)=k^{nit}u_g.$$ The definition of KMS states and a direct calculation then show the following lemma.

\begin{lem}\label{noKMS}
There is no KMS$ _\beta$ state over $\mathcal{B}$ if $\beta\neq0$.
\end{lem}

\proof If $\phi$ is a KMS$ _\beta$ state over $\mathcal{B}$, then by the property of the KMS state we have $$1=\phi(u_g*u^*_g)=\phi(u^*_g*\sigma_{i\beta}(u_g))=k^{-n\beta}.$$
This is absurd if $\phi\neq0$. So all the KMS$ _\beta$ states vanish on $\mathcal{B}$.
\endproof

The next lemma characterizes the fixed points of $Y$ under the action of $\Gl^+(\Q)$.
\begin{lem}\label{fixedpt}
Let $g\in\Gl^+(\Q)$ and $g\neq I$. If $gy=y$ with $y\in Y=\PP\times\M(\hat{\mathbb{Z}})$, then $y=(x,0)$, $x\in\PP$.
\end{lem}
\begin{proof}Since $\PP=\mathrm{PGL}_n^+(\R)\subset\mathrm{PGL}_n(\R)=\Gl(\R)/\R^*$, for $g\in\Gl^+(\Q)$, $x\in\PP$, $gx=x$ only if $g=rI$ where $r\in \R^*$ (actually $r\in \Q^*$ here).
However, for $h\in\M(\A)$, $rh=h$ only if $r=1$ or $h=0$. By assumption $g\neq I$, so $r\neq1$. This means we must have $h=0$.
So when $gy=y$ where $y=(x,h)\in Y$, we obtain $y=(x,0)$.\end{proof}

Let $$E:\,\mathcal{A}\rightarrow C_0(\Gamma\backslash Y),\ \ \ \  E(f)(y)=f(I,y)$$ be the canonical conditional expectation. The Poposition 2.1 in \cite{LLN} implies that, when the action of $\Gl^+(\Q)$ is free, the map $$\mu\mapsto\varphi(f)=\int_{\Gamma\backslash Y}E(f)(y)\,\,d\nu(y)$$ gives a one-to-one correspondence between the KMS$_\beta$ weights over $\mathcal{A}$ and
 the Radon measures $\mu$ on $Y$ satisfying the scaling condition
\begin{equation}\label{muscale}
\mu(gB)=\det(g)^{-\beta}\mu(B)
\end{equation}
for $g\in\Gl^+(\Q)$ and Borel set $B\subset Y$ with $gB\subset Y$.

The scaling condition \eqref{muscale} also implies that $\mu$ is a $\Gamma$-invariant measure on $Y$. It determines the measure $\nu$ on $\Gamma\backslash Y$ such that $$\int_Yfd\mu=\int_{\Gamma\backslash Y}\sum_{x\in\Gamma y}f(x)d\nu([y]).$$ For a measure $\mu$ satisfying the scaling condition \eqref{muscale}, the corresponding KMS$_\beta$ weight is a state iff the induced measure $\nu$ on $\Gamma\backslash Y$ is a probability measure, i.e. $\nu(\Gamma\backslash Y)=1$. In the paper, sometimes to save notation, we still use $\mu$ for the induced measure on $\Gamma\backslash Y$ if there is no otherwise meaning in context.

As we have seen in Lemma~\ref{fixedpt} above, the action of $\Gl^+(\Q)$
in the $\mathrm{GL}_n$-system is not totally free, but the fixed points are all contained in
$\PP\times\{0\}$. So combining the previous two lemmata we still have the following theorem.

\begin{thm}\label{thmKMS}
For the $\Gl$-Connes-Marcolli system, when $\beta\neq0$, there is a one-to-one correspondence between KMS$_{\beta}$-states over $\mathcal{A}$ and Radon measures $\mu$ on $Y$ such that $\mu$ satisfies the scaling condition \eqref{muscale} and induces a probability measure $\nu$ on $\Gamma\backslash Y$.
\end{thm}

\proof Let $\mathcal{J}$ be the ideal $\mathrm{C}^*_r(\Gamma\backslash G\boxtimes_\Gamma(Y-(\PP\times\{0\})))$. Notice that if a KMS$ _\beta$ state $\phi$ on $\mathcal{A}$ vanishes on $\mathcal{J}$ then it is a KMS state over $\mathcal{A}/\mathcal{J}=\mathcal{B}$. Then by Lemma 3, we see that $\phi$ vanishes on $\mathcal{A}$. Thus, given a KMS state $\varphi$ on $\mathcal{A}$, if we extend the KMS state $\varphi|_{\mathcal{J}}$ to a KMS state $\phi$ on $\mathcal{A}$, then $\phi-\varphi$ vanishes over $\mathcal{A}$. Thus, the KMS$ _\beta$ states on $\mathcal{A}$ are totally determined by the KMS$ _\beta$
states on $\mathcal{J}$.

By Lemma 4 we see that the action of $\Gl^+(\Q)$ on $Y-(\PP\times\{0\})$ is free. So by the discussion before the theorem, there is a one-to-one correspondence between KMS$_{\beta}$-states and Radon measures on $Y$ that satisfy the scaling condition \eqref{muscale}.
\endproof

\begin{rem} In the theorem above, the only restriction on $\beta$ is nonzero.
However, for the $\Gl$-system, the KMS states do not exist for all $\beta$'s. Namely, due to the 1-1 correspondence between the KMS states and the Radon measures above,
Radon measures satisfying the scaling condition \eqref{muscale} only exist  for a certain
range of values of the parameter $\beta$. This will be explained in the next section.\end{rem}

\section{The Phase Transition}
\subsection{Hecke Pairs} To prove our main theorem, we extend the approach and the results of \cite{LLN}
to the case of the $\Gl$ groups. First we need to recall some concepts of Hecke Pairs and some related formulas.

We recall the following notions from \cite{K}, I. A Hecke pair is a pair of groups $(G,H)$ such that $H\subset G$ and for any $g\in H$ the set $H/(g^{-1}Hg\cap H)$ is finite.  From \cite{K}, I. lemma 3.1, we also see $\#(H/(g^{-1}Hg\cap H))=\#(H\backslash HgH)$. If $f$ is an $H$-invariant function over a space $X$ with a $G$ action, then for the Hecke pair we define the Hecke operator $T_g$ for $g\in G$ to be
$$T_gf(x)=\frac{1}{\#(H/(g^{-1}Hg\cap H))}\sum\limits_{h\in H\backslash HgH}f(hx).$$

In this paper, we mainly focus on the Hecke pair $(G,\,\Gamma)=(\Gl^+(\Q),\,\Sl(\Z))$ (\cite{A}, Lemma 3.3.1. or \cite{K}, V. Corollary 5.3).

If the measure $\mu$ satisfies the scaling conditions \eqref{muscale}, one has $$\int_{\Gamma\backslash Y}fd\nu=\det(g)^{-\beta}\int_{\Gamma\backslash Y}T_gfd\nu,$$ where $\nu$ is the induced measure by $\mu$. (see \cite{LLN}, Lemma 2.6 and the comment after the lemma).

Also, if we let $Y\subset X$ for some $X$ with a free $G$-action, and let $Z\subset Y$ such that if $h\in G$ and $hz\in Z$ for some $z\in Z$ then $h\in\Gamma$, then once we have some $g\in G$ with $gZ\subset Y$ the following formula holds
\begin{equation}\label{eqmes}\nu(\Gamma\backslash\Gamma g\Gamma Z)=\det(g)^{-\beta}\#(H/(g^{-1}Hg\cap H))\nu(\Gamma\backslash\Gamma Z).\end{equation}
This formula can be seen in a direct calculation or by \cite{LLN}, Lemma 2.7.

\subsection{Phase Transition}

Because of the one-to-one correspondence between the KMS states on $\mathcal{A}$ and the Borel measures $\mu$ with scaling condition \eqref{muscale}, instead of studying the structure of the KMS, we rather study the properties of the corresponding Borel measures on $Y$. From the definition of the algebra $\mathcal{A}$, we see it is not very convenient to work directly with the measure on $Y$. For instance, the group action on $Y$ is only partially defined: for $g\in G$ and $y\in Y$, we cannot always guarantee that $gy\in Y$. So we want to extend the measure to a larger space to make the discussion more convenient. For this purpose, we need a lemma from \cite{LLN}.

\begin{lem}\label{ext}(\cite{LLN}, Lemma 2.2) Let $X$ be a space with a free $G$-action and $Y\subset X$ a clopen subset. Given a measure $\mu$ on $Y$ satisfying condition \eqref{muscale}, then we can uniquely extend $\mu$ to a Radon measure on $GY=\{gy\mid\,g\in G,\,y\in Y\}\subset X$
which also satisfies the scaling condition \eqref{muscale} for any Borel set $B\subset GY$.\end{lem}

Back to our specific case of the $\Gl$-Connes-Marcolli system, where $X=\PP\times\M(\A)$ and $Y=\PP\times\M(\Zh)$, the group $G=\Gl^+(\Q)$ does not act on $X$ freely, but it acts on the space $X-(\PP\times\{0\})$ freely. Once we have a Radon measure $\mu$ on $X$ satisfying the scaling condition \eqref{muscale} for any Borel set $B\subset X$, the same diagonal matrix argument (see the proof of Lemma \ref{noKMS}) shows $\mu(\PP\times\{0\})=0$. Then if we note that $GY=X$ and $G\big(Y-(\PP\times\{0\})\big)=G\big(X-(\PP\times\{0\})\big)$, by applying the Lemma \ref{ext} above, every Radon measure on $Y$ satisfying the scaling condition \eqref{muscale} can be uniquely extended to a Radon measure on $X$, which is also satisfies the scaling condition \eqref{muscale}.

With the discussion above, we can restate Theorem \ref{thmKMS} in the following way.
\begin{thmbis}{thmKMS}For the $\Gl$-Connes-Marcolli system, when $\beta\neq0$, there is a one-to-one correspondence between KMS$_{\beta}$-states over $\mathcal{A}$ and Radon measures $\mu$ on $\PP\times\M(\A)$ such that $\mu$ satisfies the scaling condition \eqref{muscale} and induces a probability measure $\nu$ on $\Gamma\backslash Y$.\end{thmbis}

So from now on, we can just study the measures on $X=\PP\times\M(\A)$. This is more convenient, because $X$ admits a well defined $G$-action, instead of having only a partially defined action.

\begin{lem}\label{Key}(Key lemma) Let $p$ be a fixed prime number and $\mu_p$ be a $\Gamma$-invariant measure on the space $\PP\times\M(\Q_p)$, which induces a probability measure on $\Gamma\backslash\PP\times\M(\Zh)$, such that $\mu_p(\PP\times\{0\})=0$ and $\mu_p(gB)=\det(g)^{-\beta}\mu_p(B)$, where $g\in\M(\Z)$ and in addition $\det(g)=p^l$, and $B$ is a Borel set.

Then, for $\beta\neq0, 1, 2, \ldots, n-1$, we have that $\PP\times \Gl(\Q_p)$ is a subset of full measure.
\end{lem}
\begin{proof} First, we need to show that, for any nonzero singular $A\in\M(\Q_p)$, there are matrices $B\in\Sl(\Z_p)$ and $C\in\Gl(\Z_p)$ such that $$BAC=\begin{pmatrix}0\\&\ddots\\&&0\\&&&p^{k_1}\\&&&&\ddots\\&&&&&p^{k_l}\end{pmatrix}$$ where $0<l<n$ and $k_{i-1}\leq k_i$. Moreover, the expression on the right hand side is unique.

For any $A=(a_{ij})_{n\times n}$, we see that each entry is the form of $a_{ij}=ap^k$ where $a$ is invertible in $\Z_p$ or $0$. So there is a smallest $p^k$ such that $ap^k$ is nonzero. Then by using the elementary matrices of exchanging rows and columns and of changing signs, we can make the smallest $ap^k$ to be the entry at $(1,1)$. For simplicity purposes, we still call the result matrix $A$.

If there is another entry element in the first column, for example, we can assume the entry $a_{21}=bp^l$ in $A=(a_{ij})$. Remember that now $a_{11}=ap^k$, if we multiply the elementary matrix $E_{21}=(e_{ij})$, in which $e_{ii}=1$, $e_{21}=-a^{-1}bp^{l-k}$ and $e_{ij}=0$ for any other entries, to the left of $A$, it makes the entry $a_{21}$ $0$. Notice that $\det(E_{21})=1$ and because of the minimality of $ap^k$, $l-k\geq0$, so $-a^{-1}bp^{l-k}\in\Z_p$. Hence $E_{21}\in\Sl(\Z_p)$. By doing this, we make all the remaining entries in the first column $0$. Similarly, by multiplying elementary matrices on the right, we use $ap^k$ to kill other nonzero entries in the first row.
By multiplying by the diagonal matrix $\mathrm{diag}(a^{-1},1,\cdots,1)$ on the right, we cancel out $a$. That is to say, we can find $B_1\in\Sl(\Z_p)$ and $C_1\in\Gl(\Z_p)$ such that $$B_1AC_1=\begin{pmatrix}p^k&0&\ldots&0\\0&a_{22}&\ldots&a_{2n}\\\vdots&\vdots&\ddots&\vdots\\0&a_{n2}&\ldots&a_{nn}\end{pmatrix}.$$

By iterating this procedure, finally we get a diagonal matrix whose entries are all either powers of $p$ or zero. Then, by changing rows, columns and signs again, we obtain a matrix in the desired form. The uniqueness is a direct application of the Cauchy-Binet formula.

By the previous argument, we have actually shown that the set of (nonzero) singular matrices in $\M(\Q_p)$ is a disjoint union of the sets $Z_{k_1\ldots k_l}$, where $$Z_{k_1\ldots k_l}=\Sl(\Z_p)\begin{pmatrix}0\\&\ddots\\&&0\\&&&p^{k_1}\\&&&&\ddots\\&&&&&p^{k_l}\end{pmatrix}\Gl(\Z_p),\quad 1\leq l\leq n-1.$$

Let $\mathscr{B}_\Gamma$ be the $\sigma$-field of $\Gamma$-invariant Borel sets in $\M(\Q_p)$. We now define a measure on $\M(\Q_p)$ by $v(B)=\mu_p(\Gamma\backslash\PP\times B)$ for $B\in\mathscr{B}_\Gamma$. To show $\PP\times \Gl(\Q_p)$ is of full measure is the same as showing that its complement has measure 0. We have shown that the complement set is covered by $\PP\times Z_{k_1\ldots k_l}$'s and that the set of $Z_{k_1\ldots k_l}$'s is countable, so we only need to show that each $\mu_p(\PP\times Z_{k_1\ldots k_l})=0$. Namely, it is enough to show that these $Z_{k_1\ldots k_l}$ have $v$-measure zero.

To do this we need to give a description of the right cosets of $$\Sl(\Z)\backslash\Sl(\Z)T_l\Sl(\Z),$$ where $T_l$ stands for the diagonal matrix $$ T_l= \mathrm{diag}(\overbrace{1,\ldots,1}^{n-l},\underbrace{p,\ldots,p}_l). $$ This is actually done in \cite{K}, V. 7. In \cite{K}, V. Proposition 7.2, it is shown that $\Gl(\Z)\backslash\Gl(\Z)T_l\Gl(\Z)$ has a set of representatives given by lower triangular matrices
$$\begin{pmatrix}p^{k_1}\\a_{ij}&\ddots\\&&p^{k_n}\end{pmatrix}$$ in which, $k_i=0$ or $1$ and $\sum k_i=l$, and
$0\leq a_{ij}<p^{k_j(1-k_i)}$.

We show that these are also representatives in $\Sl(\Z)\backslash\Sl(\Z)T_l\Sl(\Z)$. First notice that, for any $g\in\Gl(\Z)$, we have $\det(g)=\pm1$. If there are $g_1, g_2\in\Gl(\Z)$ such that $$g_1T_lg_2=\begin{pmatrix}p^{k_1}\\a_{ij}&\ddots\\&&p^{k_n}\end{pmatrix},$$ since $\det(T_1)$ and the determinant of the matrix on the right side are both positive, then
we have $\det(g_1)=\det(g_2)=\pm1$. If $\det(g_1)=\det(g_2)=-1$, we replace $g_1, g_2$ by $g_1F, Fg_2$ where $F=\mathrm{diag}(-1,1,\ldots,1)$. We see that, since $T_l$ is diagonal, $FT_lF=T_l$. Hence $$g_1FT_lFg_2=g_1T_lg_2.$$ However, this time we have $\det(g_1F)=\det(Fg_2)=1$. Namely $g_1F, Fg_2\in \Sl(\Z)$. Thus, the matrices $$\begin{pmatrix}p^{k_1}\\a_{ij}&\ddots\\&&p^{k_n}\end{pmatrix}$$
are representatives of $\Sl(\Z)\backslash\Sl(\Z)T_l\Sl(\Z)$ and they are not equivalent. Indeed, since $\Sl(\Z)$ is a
subgroup of $\Gl(\Z)$, if two matrices above are equivalent under $\Sl(\Z)$, then they will also be equivalent under $\Gl(\Z)$,
but from \cite{K}, V. Proposition 7.2 we already know those matrices are inequivalent representatives in $\Gl(\Z)\backslash\Gl(\Z)T_l\Gl(\Z)$.

To prove that all the $v(Z_{k_1\ldots k_l})=0$, we only need to show that $v(Z_{k_1,\ldots,k_{n-1}})=0$, then view the diagonal elements $0=p^{-\infty}$ in $Z_{k_1\ldots k_l}$ when $l<n-1$. 

Let $n_l=\#(\Sl(\Z)\backslash\Sl(\Z)T_l\Sl(\Z))$ and let $$E_l(x_1,x_2,\ldots,x_n)=\sum\limits_{1\leq i_1<i_2<\ldots<i_l\leq n} x_{i_1}x_{i_2}\cdots x_{i_l}$$ be the $l$-th elementary symmetric polynomial. Then we see that $$n_l=E_l(p,p^2,\ldots,p^n)/p^{l(l+1)/2},$$
by counting the representatives in the form above.

Let $f_{k_1,\ldots,k_{n-1}}=\mathbbm{1}_{Z_{k_1,\ldots,k_{n-1}}}$ be the characteristic function of the set $Z_{k_1,\ldots,k_{n-1}}$. By letting the representatives act on $Z_{k_1,\ldots,k_{n-1}}$, by the definition of the Hecke operater $T_{T_1}$ (see 3.1), we see that $$T_{T_1}f_{k_1,\ldots,k_{n-1}}=\frac{1}{n_1}(p^{n-1}f_{k_1,\ldots,k_{n-1}}+\triangle_1),$$ where $\triangle_1$ is some linear combination of other $f_{k_1,\ldots,k_{l}}$'s. By iterating this procedure, we have
\begin{eqnarray*}T_{T_1}f_{k_1,\ldots,k_{n-1}}&=&\frac{1}{n_1}(p^{n-1}f_{k_1,\ldots,k_{n-1}}+\triangle_1)\\
T_{T_2}f_{k_1,\ldots,k_{n-1}}&=&\frac{1}{n_2}(p^{n-2}\triangle_1+\triangle_2) \\
T_{T_3}f_{k_1,\ldots,k_{n-1}}&=&\frac{1}{n_3}(p^{n-3}\triangle_2+\triangle_3)\\&\cdots\\
T_{T_{n-1}}f_{k_1,\ldots,k_{n-1}}&=&\frac{1}{n_{n-1}}(p\triangle_{n-2}+f_{k_1-1,k_2-1,\ldots,k_{n-1}-1}).
\end{eqnarray*}
Then by using the integral formula $\int T_gfdv=\det(g)^\beta\int fdv$, we get
\begin{eqnarray*}n_1p^\beta v(Z_{k_1,\ldots,k_{n-1}})&=&p^{n-1}v(Z_{k_1,\ldots,k_{n-1}})+v(\triangle_1)\\
n_2p^{2\beta} v(Z_{k_1,\ldots,k_{n-1}})&=&p^{n-2}v(\triangle_1)+v(\triangle_2) \\
n_3p^{3\beta} v(Z_{k_1,\ldots,k_{n-1}})&=&p^{n-3}v(\triangle_2)+v(\triangle_3)\\&\cdots\\
n_{n-1}p^{(n-1)\beta} v(Z_{k_1,\ldots,k_{n-1}})&=&pv(\triangle_{n-2})+p^{n\beta}v(Z_{k_1,\ldots,k_{n-1}}),
\end{eqnarray*} where $v(\triangle_i)$ means the value of the integral.

We set $z=v(Z_{k_1,\ldots,k_{n-1}})$ and we cancel out all the $v(\triangle_i)$'s recursively. We then
have an equation
\begin{eqnarray*}n_{n-1}p^{(n-1)\beta}z&-&p^{n\beta}z-p\bigg(n_{n-2}p^{(n-2)\beta}z\\&-&p^2\big(n_{n-3}p^{(n-3)\beta}z-\ldots-p^{n-2}(n_1p^\beta z-p^{n-1}z)\big)\bigg)=0.\end{eqnarray*}
If $z\neq0$, we can divide out $z$ and let $x=p^\beta$. Then we get an equation of $x$:
\begin{eqnarray*}-x^n+n_{n-1}x^{n-1}&-&pn_{n-2}x^{n-2}+p\cdot p^2n_{n-3}x^{n-3}\\&-&\cdots+(\pm1)p\cdot p^2\cdots p^{n-2}n_1x\pm p\cdot p^2\cdots p^{n-1}=0.\end{eqnarray*}
Recall that the classical Vieta's formulas for polynomials states,
$$\prod_{i=1}^n(x-\lambda_i)=\Sigma_{i=1}^{n}\pm E_{n-i}(\lambda_1,\lambda_2,\ldots,\lambda_n)x^i.$$
In our polynomial above, the coefficient in front of $x^l$ is $pp^2\dots p^{n-1-l}n_l$ with a possible positive or negative sign as in the alternative sum. By the definition of $n_l$, we have
\begin{eqnarray*}pp^2\dots p^{n-1-l}n_l&=&p^{(n-1-l)(n-l)/2}E_l(p,p^2,\cdots,p^n)/p^{l(l+1)/2}\\&=&p^{(n-1-l)(n-l)/2}\big(\sum\limits_{1\leq i_1<\dots<i_{n-l}\leq n} p^{n(n+1)/2}/p^{i_1+i_2+\ldots+i_{n-l}}\big)/p^{l(l+1)/2}\\&=&\sum\limits_{1\leq i_1<\dots<i_{n-l}\leq n}
p^{n(n+1)/2+(n-1-l)(n-l)/2-l(l+1)/2-(i_1+i_2+\ldots+i_{n-l})}\\&=&\sum\limits_{1\leq i_1<\dots<i_{n-l}\leq n}p^{n(n-l)-(i_1+i_2+\ldots+i_{n-l})}
\\&=&\sum\limits_{1\leq i_1<\dots<i_{n-l}\leq n}p^{(n-i_1)+(n-i_2)+\ldots+(n-i_{n-i})}\\&=&E_{n-l}(1,p,\dots,p^{n-1}).\end{eqnarray*}

So according to the Vieta's formula, we have $p^\beta=x=1, p,\ldots, \mathrm{or}\,p^{n-1}$. However, this is ruled out by the
assumption that $\beta\neq0, 1, 2, \ldots, n-1$ in the lemma. Thus we must have $z=0$. Whence $v(Z_{k_1,\ldots,k_{n-1}})=0$, we have all $v(Z_{k_1,\ldots,k_l})=0$. Then $\mu_p(\PP\times\{g\in\M(\Q_p)|\det(g)=0\})=v(\{g\in\M(\Q_p)|\det(g)=0\})=0$. Thus,
the set $\PP\times \Gl(\Q_p)$ is a subset of full measure with respect to $\mu_p$.

\end{proof}

From the lemma, we can also show the following.
\begin{cor}\label{Keycor} Let $\mu$ be a measure on $\PP\times\M(\A)$, which induces a probability measure on $\Gamma\backslash\PP\times\M(\hat{\Z})$, such that $\mu(gB)=\det(g)^{-\beta}\mu(B)$ for $g\in\Gl^+(\Q)$ and any Borel set $B$ of $\PP\times\M(\A)$. Then the set $$\PP\times\M'(\A)=\{(x,(m_p))\in\PP\times\M(\A)\mid m_p\in\Gl(\Q_p)\}$$ has full measure.
\end{cor}
\begin{proof}
For any $\Gamma$-invariant Borel set $B\subset \M(\Q_p)$, the measure $\mu$ restricts to $\Gamma\backslash\PP\times B\times\prod_{q\neq p}\M(\Z_q)$ and gives a measure $v_p$ on $\M(\Q_p)$. Notice that the matrix $g=\mathrm{diag}(p,\ldots,p)$ is invertible in $\M(\Z_q)$, for $q\neq p$. Thus,  $$g(\PP\times\{0\}\times\prod_{q\neq p}\M(\Z_q))=\PP\times\{0\}\times\prod_{q\neq p}\M(\Z_q).$$
Since $\det(g)\neq1$ and the set $\PP\times\{0\}\times\prod_{q\neq p}\M(\Z_q)$ is $\Gamma$-invariant, the scaling condition of the measure implies that $$\mu(\Gamma\backslash\PP\times\{0\}\times\prod_{q\neq p}\M(\Z_q))=0.$$
One then has $v_p(\M(\Z_p))=\mu(\Gamma\backslash\PP\times M(\hat{\Z}))=1$. Thus, $v_p$ is a measure that satisfies all the conditions of the auxiliary measure $v$ constructed in the previous lemma. Thus, the set $$\{m\in\M(\Q_p)\mid m\not\in\Gl(\Q_p)\}$$ has zero $v_p$-measure and by the definition of $v_p$ and the scaling condition on $\mu$ together with the fact $\Gl^+(\Q)(\M(\Zh))=\M(\A)$ (see in the proof of proposition \ref{Gibbsprop})
$$\mu(\{(x,(m_q))\in\PP\times\M(\A)|m_p\not\in\Gl(\Q_p)\})=0.$$
If we set $$Z_p=\{(x,(m_q))\in\PP\times\M(\A)|m_p\not\in\Gl(\Q_p)\},$$ then $\PP\times\M'(\A)$ is
the complement of the union of all the $Z_p$'s
in $\PP\times\M(\A)$.
\end{proof}

Now let us discuss the nonexistence of the KMS states. In order to do this, we need the following definition (also see Definition 2.8 of \cite{LLN}).

\begin{defn} Let $\beta\in\R$ and let $S$ be a semigroup such that $\Gamma\subset S\subset G$. We define $$\zeta(\Gamma,S,\beta)=\sum\limits_{g\in\Gamma\backslash S}\det(g)^{-\beta}.$$
The datum $(\Gamma,S,\beta)$ is {\em summable} if $\zeta(\Gamma,S,\beta)<\infty$.
\end{defn}

Recall also that
a lower triangular integral matrix $R=(r_{ij})_{n\times n}$ is called {\em reduced} if $0\leq r_{ij}\leq r_{jj}$, $i\geq j$.

\begin{prop}(\cite{A} Lemma 3.2.7 and Exercise 3.2.10)
The following identity of sets holds: $$\Sl(\Z)\backslash\{M\in\M(\Z)\mid\,\det(M)=l\}=\{R\in\M(\Z)\mid\, R\ \ \mathrm{reduced}, \det(R)=l\}$$ for any positive integer $l$.
\end{prop}

With the help of the proposition above, we can calculate $\zeta(\Gamma,S,\beta)$ for some special choices of the semigroup $S$. Let $$\M(p)=\{M\in\M(\Z)\mid \det(M)=p^l\}$$ for some prime $p$.

\begin{lem}\label{sumS}
The datum $(\Gamma,\M(p),\beta)$ is summable only if $\beta>n-1$. In this range, the sum is given by
\begin{equation}\label{zetaMp}
\zeta(\Gamma,\M(p),\beta)=\dfrac{1}{(1-p^{-\beta+n-1})(1-p^{-\beta+n-2})\cdots(1-p^{-\beta})}.
\end{equation}
For $\beta>n$ we have
\begin{equation}\label{zetaMZ}
\zeta(\Gamma,\M^+(\Z),\beta)=\zeta(\beta-n+1)\zeta(\beta-n+2)\cdots\zeta(\beta),
\end{equation}
where $\zeta(s)$ is the Riemann zeta-function.
\end{lem}

\proof
If the entry-$(j,j)$ of a reduced matrix is $p^l$, then the element in the $j$-th column under $p^l$ only has $p^l$ choices. So the $j$-th column totally has $p^{l(n-j)}$ different cases if the entry-$(j,j)$ is $p^l$. This gives us a way of computing
\begin{eqnarray*}\zeta(\Gamma,\M(p),\beta)&=&\sum\limits_{k_1,k_2,\dots,k_n=0}^\infty p^{-\beta(k_1+k_2+\dots+k_n)}p^{k_1(n-1)}p^{k_2(n-2)}\cdots p^{k_{n-1}}\\&=&\sum\limits_{k_1,k_2,\dots,k_n=0}^\infty p^{k_1\big((n-1)-\beta\big)}p^{k_2\big((n-2)-\beta\big)}\cdots p^{k_{n-1}(1-\beta)}p^{-\beta k_n}.\end{eqnarray*}
Then sum converges only if $\beta>n-1$. When $\beta>n-1$, we can sum it up and we obtain \eqref{zetaMp}.
By the Euler product formula of the Riemann zeta function, $\zeta(s)=\prod\limits_p\frac{1}{1-p^{-s}}$, for $\mathfrak{Re}(s)>1$ and by the same counting method, when if $\beta>n$ we obtain \eqref{zetaMZ}.
\endproof

\begin{prop}\label{PolarProp}\emph{(Polarization formula: \cite{LLN}, Lemma 2.9)}. Let $(G,\Gamma)$ be a Hecke pair. Also let $\mu$ be the measure on $Y$ satisfying the scaling condition \eqref{muscale}
for some $\beta$.
Let $\nu$ be the induced measure on $\Gamma\backslash Y$. If there is a semigroup $S$
such that $\Gamma\subset S\subset G$ and $(\Gamma,S,\beta)$ is summable, and there is a $\Gamma$ invariant subset $Y_0\subset Y$ such that $gY_0\cap Y_0=\emptyset$ if $g\in G-\Gamma$ and the set $SY_0$ is conull with respect to $\mu$, then for any $S$-invariant function $f\in L^2(\Gamma\backslash Y,\,d\nu)$, we have
\begin{equation}\label{Polar1}\int_{\Gamma\backslash Y}fd\nu=\zeta(\Gamma,S,\beta)\int_{\Gamma\backslash Y_0}fd\nu.\end{equation}
As a consequence, if $P$ is the projection operator from $L^2(\Gamma\backslash Y)$ to its subspaces of $S$-invariant functions, then the following projection formula holds,
\begin{equation}\label{Polar2}Pf=T_S(f),\end{equation}
where $T_S$ is the Hecke operator given by the formula
$$T_S(f)(x)=\frac{1}{\zeta(\Gamma,S,\beta)}\sum\limits_{g\in\Gamma\backslash S/\Gamma}\det(s)^{-\beta}\#(\Gamma/(g^{-1}\Gamma g\cap\Gamma))T_gf(x).$$
\end{prop}

When $\beta\neq0,1,\cdots,n-1$, for some prime $p$ we take $$Y_p=\PP\times\Gl(\Z_p)\times\prod\limits_{q\neq p}\M(\Z_q).$$
Lemma \ref{Key} implies this set $Y_p$ has the correct properties of the set $Y_0$ described in Proposition \ref{PolarProp}
with respect to the semigroup $S=\M(p)$.

If $J$ is a finite set of primes, say $J=\{p_1,p_2,\dots,p_l\}$, $l<\infty$, we let $\M(J)$ be the set $\{M\in\M(\Z)\,\mid\,\det(M)=p_1^{a_1}p_2^{a_2}\dots p_l^{a_l}\}$. Whence we set $$Y_J=\PP\times\prod\limits_{p\in J}\Gl(\Z_p)\times\prod\limits_{q\notin J}\M(\Z_q). $$
Corollary \ref{Keycor} shows the set $Y_J$ has the correct properties of the set $Y_0$ described in Proposition \ref{PolarProp} for the semigroup $S=\M(J)$.

By the correspondence of the KMS$ _\beta$ states and the Radon measures on $Y$ with the scaling condition \eqref{muscale}, if we take $f$ as some positive constant function and let $\nu$ be the induced probability measure on $\Gamma\backslash Y$, formula \eqref{Polar1} shows $$1=\nu(\Gamma\backslash Y)=\zeta(\Gamma,\M(p),\beta)\nu(\Gamma\backslash Y_p).$$ However, this makes sense only if $(\Gamma,\M(p),\beta)$ is summable. Thus, by the discussion above, we can summarize the result as the following statement.

\begin{prop}For the $\Gl$-Connes-Marcolli system, when $\beta\neq0,1,\ldots,n-1$, the KMS$ _{\beta}$ states exist only if $\beta>n-1$.\end{prop}

To prove our main theorem we also need the following theorem.
\begin{thm}\emph{(Real approximation theorem, \cite{M})}. Let $G$ be a connected algebraic group over $\Q$. Then $G(\Q)$ is dense in $G(\R)$, where $G(\R)$ is the real Lie group. In particular, $\Gl(\Q)$ is dense in $\Gl(\R)$.\end{thm}

Now we want to show the uniqueness of the KMS state when $n-1<\beta\leq n$. Because of the one-to-one correspondence between the KMS states and the Borel measures satisfying the conditions in Theorem~\ref{thmKMS}, we only need to show the uniqueness of the measure. To show the uniqueness of the measure, a standard method is to use an argument based on ergodicity (for a full discussion, we refer the readers to \cite{LLN}, Proof of Theorem 4.2). This main idea in this type of argument is that the measures with the desired properties in Theorem~\ref{thmKMS}
form a (Choquet) simplex and the ones with ergodic action are the vertices. If the group action with respect to every measure is ergodic, then the simplex must be made of just one point. We also need to mention a common technique used in proving ergodicity: a group $G$ acts on a probability space $(X,\,\mu)$ ergodically if and only if the subspace of all $G$-invariant functions in $L^2(X)$ consists of constant functions.

\begin{thm}\label{main}
Let $n-1<\beta\leq n$. Let $\mu$ be a Borel measure on $\PP\times\M(\A)$ and $\nu$ be the corresponding measure on $\Gamma\backslash\PP\times\M(\A)$ induced by $\mu$, such that $\mu$ satisfies the scaling condition \eqref{muscale} with respect to $\beta$ for any Borel set $B$ and $\nu$ is a probability measure.
The action of $\Gl^+(\Q)$ on $\PP\times\M(\A)$ is ergodic with respect to the measure $\mu$.
\end{thm}

\begin{proof} We use the same strategy as in \cite{LLN}. The proof will be done in two steps.

First, we show that $\Gl^+(\Q)$ acts on $(\PP\times(\M(\A)/\Gl(\Zh)), \mu)$ ergodically. Here and in the following, the quotients always should be interpreted from the measure-theoretic point of view. The action of $\Gl(\Zh)$ on $\PP\times(\M(\A)$ from the right by multiplying to the factor $\M(\A)$ is compatible with the left action of $\Gl^+(\Q)$. So to say the $\Gl^+(\Q)$-action on $(\PP\times(\M(\A)/\Gl(\Zh)), \mu)$ is ergodic
is the same as to say that the $\Gl(\Zh)$-action on the measure theoretic quotient space $(\PP\times\M(\A)/\Gl^+(\Q),\mu)$ is ergodic.

As a second step, since $\Gl(\Zh)$ is compact, the measure $\mu$ is supported only in one orbit (see \cite{Z}, corollary 2.1.13), so we can think of $\PP\times\M(\A)/\Gl^+(\Q)$ as a measure theoretic quotient space $\Gl(\Zh)/H$, for some subgroup $H$ of $\Gl(\Zh)$. Noticing that $\PP$ is actually a group, if we define the $\PP$ action on $\PP\times\M(\A)$ as a multiplier on $\PP$ from the right, then this action is compatible with the diagonal action of $\Gl^+(\Q)$ on $\PP\times\M(\A)$ from the left. The group $\PP$ acts on the quotient space $\PP\times\M(\A)/\Gl^+(\Q)$ hence on $\Gl(\Zh)/H$. The $\PP$-action on $\Gl(\Zh)/H$ is continuous and $\Gl(\Zh)$ is totally disconnected, so it is a trivial action.
If we can show that the action of $\PP$ is ergodic, then we get that all the $\Gl^+(\Q)$-invariant functions on $\PP\times\M(\A)$ are constant. Thus the $\Gl^+(\Q)$ action is ergodic.

\smallskip

\noindent {\em (1) $\Gl^+(\Q)$ acts on $(\PP\times(\M(\A)/\Gl(\Zh)), \mu)$ ergodically}.

Before showing this, we observe the fact that any continuous $\M^+(\Z)$-invariant function $f$ on $\Gamma\backslash\PP$ is constant. Thus, we can view $f$ as a $\Gamma$-invariant function on $\PP$. This is true, because for any $\Gamma$-invariant function $f$ on $\PP$ that is also $\M^+(\Z)$ invariant, we have the following facts. First, for any $g\in\Gl^+(\Q)$, there is a $k\in\M^+(\Z)$ making $kg\in\M^+(\Z)$, hence $f(gx)=f(kgx)=f(x)$. Namely, $f$ is $\Gl^+(\Q)$-invariant. Second, $\Gl(\Q)$ is dense in $\Gl(\R)$ by the real approximation theorem. It follows that $f$ is also $\Gl^+(\R)$-invariant, i.e., $f$ is constant on $\Gamma\backslash\PP$.

Let $J$ be some finite set of primes. For any bounded Borel function $f$ over $\Gamma\backslash\PP$, we define $f_J=f\mathbbm{1}_{Y_J}$ on $\PP\times\M(\Zh)$, so that $f_J(x,\rho)=f(x)\mathbbm{1}_{Y_J}(x,\rho)$, for $(x,\,\rho)\in\PP\times\M(\Zh)$. We also need two operators. Consider the Hilbert space $L^2(\Gamma\backslash(\PP\times\M(\Zh)),v)$, where $v$ is the probability measure on $\Gamma\backslash(\PP\times\M(\Zh))$ induced by $\mu$. Let $P_J$ be the projection to the subspace of $\M(J)$-invariant functions, while $P$ is the projection to the subspace of $\M^+(\Z)$-invariant functions.

For another finite set $F$ of primes disjoint from $J$, by the projection formula, we have
\begin{equation}\label{projeq}P_Ff_J=(T_Ff)_J,\end{equation}
where $T_F=T_{\M(F)}$ with respect to the action of $\Gl^+(\Q)$ on $\Gamma\backslash\PP$. Notice that $P=\lim_IP_I$, where $I$ runs through all the finite set of primes. Then $P$ factor through the space of $G_J$-invariant functions over $\Gamma\backslash\PP$ and from equation \eqref{projeq} we see that $Pf_J$ is in the form of $Pg_J$ in which $g$ is some $G_J$-invariant function. Where $$G_J=\{g\in\Gl^+(\Q)\mid\, \det(g)=p_1^{l_1}p_2^{l_2}\cdots p_k^{l_k},\,p_i\notin J\}$$ is a dense subgroup of $\Gl^+(\Q)$. Hence $G_J$ is also dense in $\Gl^+(\R)$. Thus, $g=a$ is a constant. We then have $$Pf_J=aP\mathbbm{1}_{\Gamma\backslash Y_J}$$ by the definition of the restriction operator $f \mapsto f_J$. Since $PP_J=P$ and $P_J\mathbbm{1}_{\Gamma\backslash Y_J}$ is a constant by the projection formula \eqref{Polar1}, we obtain that $$Pf_J=aP\mathbbm{1}_{\Gamma\backslash Y_J}=aPP_J\mathbbm{1}_{\Gamma\backslash Y_J}$$ is a constant.

We always regard a $\prod_{p\in J}\Gl(\Z_p)$-right-invariant function $f$ on the quotient $$\Gamma\backslash(\PP\times\prod_{p\in J}\M(\Z_p))$$ as a $\Gl(\Zh)$-right-invariant function on the quotient $\Gamma\backslash(\PP\times\M(\Zh))$.
Moreover, all such functions for all $J$'s are dense in the space of $\Gl(\Zh)$-right-invariant functions. So we can simply restrict our observation to $Pf$. To calculate $Pf$, we need the operator $P_J$ again and the fact that $PP_J=P$. From the formula \eqref{Polar2}, we see $P_J$ is compatible with the right $\prod_{p\in J}\Gl(\Z_p)$ action (seen as $\Gl(\Zh)$ action). So $P_Jf$ is still $\Gl(\Zh)$ invariant. Since $\M(J)Y_J$ is of full measure, we can think $P_Jf$ is supported in $\M(J)Y_J$. Recall the construction of $Y_J$. The middle part of $Y_J$ is exactly the group $\prod_{p\in J}\Gl(\Z_p)$. The property of $\Gl(\Zh)$ invariance implies that $P_Jf$ is only determined by the factor in $\PP$. Namely $P_Jf$ is actually in the form of $g_J$ for some $g$ on $\Gamma\backslash\PP$. So $Pf=PP_Jf=Pg_J$. By the previous paragraph, $Pg_J$ is a constant, so is $Pf$. This shows all the $\Gl^+(\Q)$ invariant functions on $$(\PP\times(\M(\A)/\Gl(\Zh)), \mu)$$ are constants. So the ergodicity follows.

\medskip

\noindent {\em (2) The group $\PP$ acts on the quotient space $\PP\times\M(\A)/\Gl^+(\Q)$ ergodically}.

Recall the group $\PP$ acts on $\PP\times\M(\A)$ at the first factor from the right. So the $\PP$ action and the $\Gl^+(\Q)$ action are compatible. We show that the action of $\Gl^+(\Q)$ on $\PP\times\M(\A)/\PP=\M(\A)$ is ergodic.

Let $P$ be the projection from the Hilbert space $L^2(\M(\Zh))$ to the subspace of $\M^+(\Z)$ invariant functions. It is enough to show that $Pf$ is a constant function for $f\in L^2(\M(\Zh))$.

For any $J$, we need to show that all the $\Gl^+(\Q)$ invariant functions on the product $\prod_J\M(\Q_p)$ are constant. Since $\prod_J\Gl(\Q_p)$ is dense in $\prod_J\M(\Q_p)$, we show that $\Gl^+(\Q)$ invariant functions on $\prod_J\Gl(\Q_p)$ are constant. This is equivalent to the fact that $\M^+(\Z)$ invariant functions on $\prod_J\Gl(\Z_p)$ are constant. Once this is true for a $J$, we can vary all the possible $J$'s. Because that all the such functions for all $J$'s span a dense subspace of the space of $\M^+(\Z)$ invariant functions, it follows that all the $\M^+(\Z)$ invariant functions are constant.

To find $Pf$, we notice that $\Gamma\subset\M^+(\Z)$, so we always can map $f$ to the subspace of $\Gamma$ invariant functions first. We therefore assume that $f$ is $\Gamma$ invariant. We define
$$f_J=f\mathbbm{1}_{\prod_J\Gl(\Z_p)},\quad \mathrm{such~that}\quad f_J((m_p)_p)=f((m_p)_p)\mathbbm{1}_{\prod_J\Gl(\Z_p)}((m_p)_p),$$
on $\M(\Zh)$, for any $\Gamma$ invariant function $f$ over $\prod_J\Gl(\Z_p)$. The group $\Gamma$ is dense in $\Sl(\Zh)$, hence $f$ only depends on the value of $\det(m)\in\prod_J\Z^*_p$. Thus, it is of the form $f(m)=\chi(\det(m))$, where $\chi$ is a character in the dual group of $\prod_J\Z^*_p$. Thus, if $\chi$ is trivial, then $f_J=\mathbbm{1}_{Y_J}$. So $Pf_J$ is a constant as we showed in step (1). When $\chi$ is not trivial, we use the projection formula \eqref{Polar2}, and we find
\begin{equation}\label{char}
\begin{split}
P_Ff_J(m)=\chi&(\det(m_p)_{p\in J})\\&\times\prod\limits_{q\in F}\frac{(1-p^{-\beta})(1-p^{-(\beta-1)})\cdots(1-p^{-(\beta-n+1)})}{(1-\chi(p)p^{-\beta})(1-\chi(p)p^{-(\beta-1)})\cdots(1-\chi(p)p^{-(\beta-n+1)})},
\end{split}
\end{equation}
where $F$ is a set of finite primes that is disjoint from $J$.\\

We see that $P=\lim_FP_F$ and $PP_F=P$. If we keep enlarging $F$, the right hand side of \eqref{char} is approaching to $0$, by the properties of Dirichlet series. So $Pf_J=0$. In both cases, we find that $Pf_J$ is constant. Then, by varying all the possible $J$'s, we can see that all the $Pf$'s are constant functions.

\end{proof}

\subsection{The existence of the KMS states} In the previous section, we have shown there is no KMS states when $\beta\in(-\infty, 0)\cup(0,1)\cup(1,2)\cup\cdots\cup(n-1,n)$ and the uniqueness of the KMS state when $n-1<\beta\leq n$. In the section, we briefly discuss how to construct the KMS states and the existence at the dividing points of $\beta=0,1,2,\ldots,n-1$.

\smallskip
\subsubsection{The case $\beta=0$}

\begin{prop}\label{beta0lem}
There are no KMS states on the $\Gl$-system for $\beta=0$.
\end{prop}

\proof
 In this case, the KMS states are tracial. If there is a trace, say $\varphi$, on $\mathcal{A}$, $\varphi$ gives a nonzero $\Gl^+(\Q)$ invariant measure on $\PP\times\M(\A)$. We can similarly define $$v(B)=\mu(\Gamma\backslash\PP\times B\times\prod_{q\neq p}\M(\Z_q))$$ for $B\subset\M(\Q_p)$ and $\Gamma$-invariant. So if there is a $Z_{k_1\ldots k_l}$ with $v(Z_{k_1\ldots k_l}\neq0)$, then $$v(Z_{k_1+1\ldots k_l+1})=v(pZ_{k_1\ldots k_l})\neq0.$$ So, $0\neq v(Z_{k_1\ldots k_l})=v(Z_{k_1+1\ldots k_l+1})=v(Z_{k_1+2\ldots k_l+2})=\cdots$. Notice all these $Z_{k_1\ldots k_l}$'s are disjoint in $\M(\Z_p)$ and
 $$v(\bigcup_{n\geq0}Z_{k_1+n\ldots k_l+n})=\sum_{n\geq0}v(Z_{k_1+n\ldots k_l+n})=\infty.$$ So this contradicts the fact that $v(\M(\Z_p))<\infty$, since $\M(\Z_p)$ is compact. Then we have $v(Z_{k_1\ldots k_l})=0$ for all $Z_{k_1\ldots k_l}$. Again, following Corollary \ref{Keycor}, $\PP\times\M'(\A)$ is of full measure. We again let $\nu$ be the measure induced by $\mu$ on $\Gamma\backslash Y$. By formulas \eqref{eqmes} and \eqref{Polar1}, we have $$\nu(\Gamma\backslash Y)=\zeta(\Gamma,\M(p),0)\nu(\Gamma\backslash Y_p).$$ This is a contradiction, because $\zeta(\Gamma,\M(p),0)=\infty$ but $\nu(\Gamma\backslash Y)$ and $\nu(\Gamma\backslash Y_p)$ are both positive finite numbers. So $\mu\equiv0$ and there is no KMS state on $\mathcal{A}$ for $\beta=0$.
 \endproof

\smallskip
\subsubsection{The cases $\beta=1,2,\ldots,n-1$}

First let $\mu$ be a Borel measure corresponding to some KMS$_\beta$ state as in Theorem \ref{thmKMS}$'$ and let $\nu$ be the induced probability measure on $\Gamma\backslash Y$.
Note that by formula \eqref{eqmes}, we have $$\nu(\Gamma\backslash\M(p)Y_p)=\zeta(\Gamma,\M(p),\beta)\nu(Y_p).$$ When $\beta=1,2,\ldots,n-1$, the series defining $\zeta(\Gamma,\M(p),\beta)$ is divergent. So $\nu(Y_p)=0$. Let $p$ run over all the primes, we conclude that the set $$\PP\times\{M=(m_p)_p\in\M(\A)\mid \det(m_p)=0\}$$ is a set of full measure. Moreover, as shown in Corollary \ref{Keycor}, the set $$\PP\times\{0\}\times\prod_{q\neq p}\M(\Z_q)$$ is always of $\mu$-measure $0$. So the set $\PP\times\M^{0,\beta}(\A)$, where $$\M^{0,\beta}(\A)=\{M=(m_p)_p\in\M(\A)\mid m_p\neq0,\,\det(m_p)=0\},$$ is of full measure with respect to $\mu$.

Moreover, actually there are a lot of such measures.
\begin{prop}\label{betaklem}
Let $\beta = k \in \{ 1, 2, \ldots, n-1 \}$ and let $\mu_k=\mu_\PP\times\mu'$, where $\mu_\PP$ is a $\Gl^+(\Q)$ invariant measure on $\PP$ and $\mu'$ is
the Haar measure on $\A^{kn}\simeq \M^k(\A)$. Here
$$\M^k(\A)=\{M\in\M(A)\mid M=(\mathbf{0}_{n-k}|N_k)\},$$ and $\mathbf{0}_{n-k}$ is the zero matrix of size $n\times(n-k)$ and $N_k\in\mathrm{Mat}_{n\times k}(\A)$. For $g\in\Gl(\Zh)$, the measure $\mu_{kg}=\mu_k(\,\cdot\, g^{-1})$
with support in $\PP\times\M^k(\A)g$ satisfies the scaling condition \eqref{muscale}, hence it defines a
KMS$_\beta$ state.
\end{prop}

\proof
As an additive group, $\M^k(\A)$ is the same as $\A^{kn}$. So the Haar measure on $\A^{kn}$ gives a measure $\mu'$ on $\M(\A)$ supported in $\M^k(\A)$. Let $\mu_k=\mu_\PP\times\mu'$. The measure $\mu_k$ also naturally satisfies the scaling condition \eqref{muscale} for $\beta=k$ (One can think of the case of the Lebesgue measure over $\R^n$ which is a Haar measure if we see $\R^n$ is an additive group). We also can normalize $\mu_k$ so that  the corresponding measure $\nu_k$ is a probability measure on $\Gamma\backslash Y$. So a KMS$_\beta$ state has been constructed. Let $g\in\Gl(\Zh)$. So $g$ acts on $\PP\times\M(\A)$ from the right by multiplying to the second factor on the right. Then we define $\mu_{kg}=\mu_k(\,\cdot\, g^{-1})$. The measure $\mu_{kg}$ is also a measure satisfying the scaling condition \eqref{muscale} with support in $\PP\times\M^k(\A)g$.
\endproof

Thus, we have obtained a construction of a non-empty set of KMS$_\beta$ states for each
$\beta = k \in \{ 1, 2, \ldots, n-1 \}$.

\smallskip
\subsubsection{The ${\rm GL}_2$-case for $\beta=1$}

Now, let us specifically turn to the case $n=2$, $\beta=1$.  In this case, $\PP=\mathrm{PGL}^+_2(\R)=\mathrm{GL}_2^+(\R)/\R^*$ and $\Gamma=\mathrm{SL}_2(\Z)$. We let $$\mathrm{Mat}_2^1(\A)'=\{B=(b_p)_p\in\mathrm{Mat}_2^1(\A)\mid b_p\neq0\}.$$ Let $M=(m_p)_p\in\mathrm{Mat}_2(\A)$ such that $m_p\neq0$ and $\det(m_p)=0$ for all $p$. We have
$$m_p=\begin{pmatrix}n_{11}&n_{12}\\n_{21}&n_{22}\end{pmatrix}\in\mathrm{Mat}_2(\Q_p).$$
Since $m_p\neq0$ and $\det(m_p)=0$, the matrix $m_p$ has rank $1$. So $$\begin{pmatrix}n_{11}\\n_{21}\end{pmatrix}=ap^\alpha\begin{pmatrix}n_{12}\\n_{22}\end{pmatrix},$$ where $a$ is invertible in $\Z_p$. We can assume $\alpha>0$, if not we multiply the matrix $$\begin{pmatrix}0&1\\1&0\end{pmatrix}\in\mathrm{GL}_2(\Z_p)$$ to the right of $m_p$ to swap the columns. Let $$g_p=\begin{pmatrix}1&0\\-ap^\alpha&1\end{pmatrix}\in\mathrm{GL}_2(\Z_p)$$ and $$m_pg_p=\begin{pmatrix}n_{11}&n_{12}\\n_{21}&n_{22}\end{pmatrix}\begin{pmatrix}1&0\\-ap^\alpha&1\end{pmatrix}=\begin{pmatrix}0&n_{12}\\0&n_{22}\end{pmatrix}\in\mathrm{Mat}_2^1(\Q_p)'.$$
Considering all $p$'s, we actually show that for any $M=(m_p)_p\in\mathrm{Mat}_2(\A)$ such that $m_p\neq0$ and $\det(m_p)=0$, there is a $g=(g_p)_p\in\mathrm{GL}_2(\Zh)=\prod_p\mathrm{GL}_2(\Z_p)$ such that $Mg\in\mathrm{Mat}_2^1(\A)'$.

 \smallskip

 Let $B=(b_p)_p\in\mathrm{Mat}_2^1(\A)'$ and $h=(h_p)_p\in\mathrm{GL}_2(\Zh)$. We want to see what happens when $Bh\in\mathrm{Mat}_2^1(\A)'$. For some $p$, we see that $$b_p=\begin{pmatrix}0&c_1\\0&c_2\end{pmatrix}$$ and $c_1,\,c_2$ can not be both $0$. If $$h_p=\begin{pmatrix}a_{11}&a_{12}\\a_{21}&a_{22}\end{pmatrix},$$ then
 $$b_ph_p=\begin{pmatrix}a_{21}c_1&a_{22}c_1\\a_{21}c_2&a_{22}c_2\end{pmatrix}.$$ If $Bh\in\mathrm{Mat}_2^1(\A)'$, this only can happen if $a_{21}=0$. Since $p$ is taken arbitrarily, we actually have shown that $h$ is an upper triangular matrix in $\mathrm{GL}_2(\Zh)$, i.e. the stabilizer of $\mathrm{Mat}_2^1(\A)'$ is the subgroup of upper triangular matrices in $\mathrm{GL}_2(\Zh)$, which is denoted by $U$. Immediately it follows that, if $$\mathrm{Mat}_2^1(\A)'M_1\bigcap\mathrm{Mat}_2^1(\A)'M_2\neq\varnothing ,$$ then $M_1=uM_2$ with $u\in U$. Combining with the previous paragraph, we have shown that $$\{\mathrm{Mat}_2^1(\A)'g\}_{g\in U\backslash\mathrm{GL}_2(\Zh)}$$ forms a partition of the set $\mathrm{Mat}_2^{0,1}(\A)$ and there is a map $$\pi:\,\mathrm{Mat}_2^{0,1}(\A)\rightarrow U\backslash\mathrm{GL}_2(\Zh)$$ such that $\pi(M)=g$ if $M\in\mathrm{Mat}_2^1(\A)'g$. We also consider this map as $$\pi:\,\PP\times\mathrm{Mat}_2^{0,1}(\A)\rightarrow U\backslash\mathrm{GL}_2(\Zh)$$ so that $\pi(x,M)=g$ for $$(x,M)\in\PP\times\mathrm{Mat}_2^{0,1}(\A),$$ if $M\in\mathrm{Mat}_2^1(\A)'g$.

 \smallskip

For $\mathrm{Mat}_2^1(\A)'g$, there is a measure $\mu_g$ supported in $\PP\times\mathrm{Mat}_2^1(\A)'g$ and defining a KMS$_\beta$ state, as we shown in Proposition \ref{betaklem} above.
We are going to show that
the action of $\mathrm{GL}_2^+(\Q)$ on $(\PP\times\mathrm{Mat}_2^1(\A)'g,\,\mu_g)$, with $g\in U\backslash\mathrm{GL}_2(\Zh)$ is ergodic.
To accomplish this, we need the help of the following important theorem in algebraic number theory (see \cite{N}, Chapter III, \S1, Exercise 1 or \cite{CF}, Chapter II, Theorem 15).

\begin{thm}(Strong Approximation Theorem) Let $\mathfrak{p}_0$ be an arbitrary place of $\Q$ and $S$ be a finite set of places such that $\mathfrak{p}_0\notin S$. Given any $(\alpha_{\mathfrak{p}})_{\mathfrak{p}\in S}\in\prod_{\mathfrak{p}\in S}\Q_{\mathfrak{p}}$ and $\epsilon>0$, there exists some $c\in\Q$ such that $|\,c-\alpha_{\mathfrak{p}}|_{\mathfrak{p}}<\epsilon$ for $\mathfrak{p}\in S$ and $|\,c\,|_{\mathfrak{p}}\leq1$ for $\mathfrak{p}\notin S, \mathfrak{p}\neq\mathfrak{p}_0$.

In particular, if we take $\mathfrak{p}_0=\infty$, this says that $\Q$ is dense in $\A$ via the diagonal embedding.
\end{thm}

We then have the following result.

\begin{prop}\label{Gl2erg}
Let $\lambda$ be a quasi-invariant measure on $\PP\times\mathrm{Mat}_2^1(\A)'$.
The action of $\mathrm{GL}^+_2(\Q)$ on $(\PP\times\mathrm{Mat}_2^1(\A)',\, \lambda)$ is ergodic. Moreover, if $\lambda$ is normalized, then such a measure is unique.
\end{prop}

Let $\mathrm{Mat}_2^1(\A)'_J$ be the image of the canonical map $$\mathrm{Mat}_2^1(\A)'\rightarrow\prod_{p\in J}\mathrm{Mat}_2^1(\Q_p)$$ for some finite set $J$ of primes over $\Q$. Note that $$\PP=\mathrm{PGL}^+_2(\R)=\mathrm{GL}_2^+(\R)/\R^*.$$ For any $$r=\begin{pmatrix}r_{11}&r_{12}\\r_{21}&r_{22}\end{pmatrix}\in\mathrm{GL}_2^+(\R) \ \ \text{ and } \ \  (m_p)_{p\in J}=\begin{pmatrix}\begin{pmatrix}0&m_1^{(p)}\\0&m_2^{(p)}\end{pmatrix}\end{pmatrix}_{p\in J}\in\mathrm{Mat}_2^1(\A)'_J, $$
by the Strong Approximation Theorem above, there is a
$$Q=\begin{pmatrix}q_{11}&q_{12}\\q_{21}&q_{22}\end{pmatrix}\in\mathrm{GL}^+_2(\Q)$$ such that
\begin{itemize}
\item[] $q_{11}$ is close enough to $(r_{11},(m^{(p)}_1)_{p\in J})$,
\item[] $q_{22}$ is close enough to $(r_{22},(m^{(p)}_2)_{p\in J})$,
\item[] $q_{12}$ is close enough to $(r_{12},(0)_{p\in J})$,
\item[] $q_{21}$ is close enough to $(r_{21},(0)_{p\in J})$,
\end{itemize}
where $(r_{11},(m^{(p)}_1)_{p\in J})$ means an element in $\R\times\prod_{p\in J}\Q_p$, and we see the rational number $q_{11}$ as an element in $\R\times\prod_{p\in J}\Q_p$ via the diagonal embedding.

If we let $e$ be the identity element in $\PP$ and $$E=\begin{pmatrix}0&1\\0&1\end{pmatrix},$$ then the product $Q(e,E)$ is close enough to the element $$(r,(m_p)_{p\in J})\in\PP\times\mathrm{Mat}_2^1(\A)'_J.$$ So the image of the product $\mathrm{GL}^+_2(\Q)(e,I)$ is dense in $\PP\times\mathrm{Mat}_2^1(\A)'_J$. We roughly see $\Gl^+(\Q)$ has a subgroup which is also dense in the group of $\Gl^+(\R)\times\prod_J\Q_p^*\times\prod_J\Q_p^*$ (by neglecting the small differences). So the $\mathrm{GL}^+_2(\Q)$-invariant functions on $\PP\times\mathrm{Mat}_2^1(\A)'_J$ are almost constant. Those functions also can be seen as functions on $\PP\times\mathrm{Mat}_2^1(\A)'$. By taking all the possible $J$'s, we find a dense subspace of $\mathrm{GL}^+_2(\Q)$ invariant functions over $\PP\times\mathrm{Mat}_2^1(\A)'$ consisting of constants.
Thus we obtain that the action of $\mathrm{GL}^+_2(\Q)$ on $(\PP\times\mathrm{Mat}_2^1(\A)',\, \lambda)$ is ergodic. Moreover,  we also see that, if $\lambda$ is normalized, for example $\lambda(\Gamma\backslash\PP\times\mathrm{Mat}_2^1(\Zh)')=1$, then such a measure is unique.
\endproof

 Let $\rho$ be a measure on $\PP\times\mathrm{Mat}_2^{0,1}(\A)$ that satisfies conditions in Theorem \ref{thmKMS}$'$. So $\mathrm{GL}^+_2(\Q)$ acts on $$(\PP\times\mathrm{Mat}_2^1(\A)'g,\rho|_{\mathrm{Mat}_2^1(\A)'g}),$$ with $g\in U\backslash\mathrm{GL}_2(\Zh)$ ergodically. If we set $\kappa=\rho\circ\pi^{-1}$, then $$\rho=\int_{U\backslash\mathrm{GL}_2(\Zh)}\mu_gd\kappa(g).$$ So we have obtained the following result.

 \begin{prop}For $n=2,\,\beta=1$, the set of extremal KMS$_1$ states on the $\mathrm{GL}_2$-Connes-Marcolli system is identified with the set of $U\backslash\mathrm{GL}_2(\Zh)$. \end{prop}

\smallskip
\subsubsection{The case $\beta\in(n-1,n]$: existence}


We have proved in Theorem \ref{main} that, in the range $\beta\in(n-1,n]$, if the set of
KMS states is non-empty, then it consists of a single point. Here we show the existence.

\begin{prop}\label{existprop}
For $\beta\in(n-1,n]$, the set of KMS$_\beta$ states is non-empty, hence by
Theorem \ref{main} it consists of a unique element.
\end{prop}

\proof
A KMS state can be constructed in the following way (\emph{cf.} \cite{LLN}, the first part of section 4). Let $n-1<\beta\leq n$. At place $p$, there is the normalized Haar measure $\mu_p$ on the compact group $\Gl(\Z_p)$ such that
$$\mu_p(\Gl(\Z_p))=(1-p^{-\beta+n-1})(1-p^{-\beta+n-2})\cdots(1-p^{-\beta}).$$
This measure can be uniquely extended (we still call it $\mu_p$) to the group $\Gl(\Q_p)$ by
$$\mu_p(K)=\sum\limits_{g\in\Gl(\Z_p)\backslash\Gl(\Q_p)}|\det(g)|_p^{-\beta}\mu(gK\cap\Gl(\Z_p)),$$
for compact $K\subset\Gl(\Q_p)$. Here $|\cdot|_p$ is the standard $p$-adic norm of $\Q_p$. Set $\mu_\PP$ to be the Haar measure of $\PP$. Since $\Gamma$ is a lattice of $\Sl(\R)$, we normalize $\mu_\PP$ so that $\mu_\PP(\Gamma\backslash\PP)=1$. Then we let $$\mu_\beta=\mu_\PP\times\prod\limits_p\mu_p.$$ The measure $\mu_\beta$ satisfies the conditions in Theorem \ref{thmKMS}($'$), hence there is a KMS at inverse temperature $\beta$.
\endproof

\smallskip
\subsubsection{The case $\beta>n$}
As before, we use the fact that
there is a one-to-one correspondence between the KMS$_\beta$ states and the Borel measures
satisfying the conditions in Theorem \ref{thmKMS}$'$. There may be many of these, like the one
constructed in the previous subsection. Using the same argument in \cite{LLN}, Remark 4.8 or \cite{CM2}, Theorem 3.97, we focus on classifying the extremal states.

\begin{prop}\label{Gibbsprop}
Let $\mu$ be a Borel measure satisfying the conditions in Theorem \ref{thmKMS}$'$, with
$\beta>n$, and $\nu$ be the induced probability measure on $\Gamma\backslash Y$.
\end{prop}

\proof
On one hand, from the formula \eqref{eqmes}, we have $$\nu(\Gamma\backslash\M^+(\Z)(\PP\times\Gl(\Zh)))=\zeta(\Gamma,\M^+(\Z),\beta)\nu(\Gamma\backslash\PP\times\Gl(\Zh)).$$ On the other hand, let $J$ be a finite set of primes. By formula \ref{Polar1}, we have $$1=\nu(\Gamma\backslash Y)=\zeta(\Gamma,\M(J),\beta)\nu(\Gamma\backslash Y_J).$$ So $\nu(\Gamma\backslash Y_J)=\zeta(\Gamma,\M(J),\beta)^{-1}$. Enlarging $J$, we finally have $$\nu(\Gamma\backslash\PP\times\Gl(\Zh))=\zeta(\Gamma,\M^+(\Z),\beta)^{-1}.$$ This shows that $$\nu(\Gamma\backslash\Gl^+(\Z)(\PP\times\Gl(\Zh)))=\zeta(\Gamma,\M^+(\Z),\beta)\nu(\Gamma\backslash\PP\times\Gl(\Zh))=1.$$ So $\M^+(\Z)(\PP\times\Gl(\Zh))$ is a subset of full measure of $\PP\times\M(\Zh)$. By multiplying $\Gl^+(\Q)$, we have  that $$\Gl^+(\Q)(\PP\times\Gl(\Zh))=\PP\times\Gl(\A)$$ is a subset of full measure of $\PP\times\M(\A)$. Here we use two facts:
\begin{enumerate}
\item $\Gl^+(\Q)\Gl(\Zh)=\Gl(\A)$,
\item $\Gl^+(\Q)\M(\Zh)=\M(\A)$.
\end{enumerate}
(1) holds because $\Gl$ (as an algebraic group over $\Q$) has class number 1.

\noindent (2) is more straightforward. Let $m=(m_p)_p\in\M(\A)$. By the definition of the Adele ring,
there are finitely many $m_p$'s lie in $\M(\Q_p)$. Each $m_p$ is an $n\times n$ matrix with entries in $\Q_p$.
So each entry has the form $p^{-k}l_p$ with $l_p\in\Z_p$. So we can find a integer $n_p$ (we also think $n_p$ as an element in $\Gl^+(\Q)$) such that $n_pm_p\in\M(\Z_p)$. Since only finitely many $m_p$'s lie in $\M(\Q_p)$, we can take the finite product of all these $n_p$'s, say $\tilde{n}$. Thus, $\tilde{n}m\in\M(\Zh)$ with $\tilde{n}\in\Gl^+(\Q)$. So $m\in\Gl^+(\Q)\M(\Zh)$. Thus, $\M(\A)\subset\Gl^+(\Q)\M(\Zh)$. Since $\Gl^+(\Q)\subset\M(\A)$ and $\M(\Zh)\subset\M(\A)$,
we also have $\M(\A)=\Gl^+(\Q)\M(\Zh)$.

\smallskip

Given a $y\in\Gamma\backslash\PP\times\Gl(\Zh)$, the point mass measure $\nu_y$ at $y$ with mass $$\zeta(\Gamma,\M^+(\Z),\beta)^{-1}$$ determines an ergodic measure $\mu_y$ on $\PP\times\Gl(\A)$ satisfying the scaling condition \eqref{muscale} such that $$\mu_y(g\Gamma y)=\det(g)^{-\beta}\mu_y(\Gamma y),$$ with $g\in\Gl^+(\Q)$. This map $y\mapsto\mu_y$ gives a one-to-one correspondence between the set $\Gamma\backslash\PP\times\Gl(\Zh)$ and the set of extremal KMS$_\beta$ states on $\mathcal{A}$ (\emph{cf.} \cite{CM2}, Theorem 3.97).
\endproof

\bibliographystyle{amsplain}

\end{document}